\documentclass[a4paper,reqno]{amsart}
\usepackage{amssymb,amscd,amsmath}
\usepackage{upref}
\usepackage[american]{babel}
\usepackage[all]{xy}
\usepackage{fancyhdr}
\usepackage{enumerate}
\usepackage{mathtools}

\theoremstyle{plain}
\newtheorem{theo}{Theorem}[section]

\newtheorem{prop}[theo]{Proposition}
\newtheorem{lem}[theo]{Lemma}
\newtheorem{coro}[theo]{Corollary}
\newtheorem*{theorem*}{Theorem}

 \theoremstyle{definition}

 \theoremstyle{remark}

 \newtheorem{rem}[theo]{Remark}

\numberwithin{equation}{section}

\DeclareMathOperator{\gr}{grad}
 
\DeclareMathOperator{\wed}{wd}

\begin{document}

\title[Free group algebras]{Free group algebras in division rings with valuation I}
\author{Javier S\'anchez}

\address{Department of Mathematics - IME, University of S\~ao Paulo,
Rua do Mat\~ao 1010, S\~ao Paulo, SP, 05508-090, Brazil}
\email{jsanchez@ime.usp.br}
\thanks{Supported by FAPESP-Brazil, Proj. Tem\'atico 2015/09162-9, by 
 Grant CNPq 307638/2015-4 and
by DGI-MINECO-FEDER 
through the grants MTM2014-53644-P and MTM2017-83487-P}



\begin{abstract}
Let $R$ be an algebra over a commutative ring $k$.
Suppose that $R$ is endowed with a descending filtration indexed
on an ordered group $(G,<)$ such that the restriction to $k$
is positive.  We 
show that the existence of free algebras on a certain set of generators in the induced
graded ring $\gr(R)$ implies the existence of free group algebras in $R$. Our best results are obtained
for division rings endowed with a valuation.
\end{abstract}

\maketitle

\setcounter{tocdepth}{1}
\tableofcontents

\section*{Introduction}
In the study of division rings, two conjectures that still remain open have attracted 
atention during the last decades. The first one was stated by A. I. Lichtman in \cite{Lichtmanonsubgroupsof}:

\begin{enumerate}[(G)]
\item If a division ring $D$ is not commutative (i.e. not a field),
then $D\setminus\{0\}$ contains a noncyclic free group.
\end{enumerate}
It holds true when the center of $D$ is
uncountable~\cite{chibafreegroupsinsidedivisionrings} and when $D$
is finite-dimensional over its center
\cite{Goncalvesfreegroupsinsubnormal}. The second conjecture was
formulated in~\cite{Makaronfreesubobjects}:
\begin{enumerate}[(A)]
\item Let $D$ be a skew field with center $Z$. If $D$ is finitely
generated as division ring over $Z$, then $D$ contains a noncommutative
free $Z$-algebra.
\end{enumerate}
Some important results have been obtained in \cite{Makar1},
\cite{Makar2}, \cite{Lichtmanfreeuniversalenveloping},
\cite{BellRogalskifreesubalgebrasoreextensions}, \cite{FerreiraFornaroliGoncalves} and other papers. In
many examples in which conjecture (A) holds, $D$ in fact contains a
noncommutative free group $Z$-algebra.  For example, this always happens
if the center of $D$ is uncountable~\cite{GoncalvesShirvani} (or \cite{SanchezObtaininggraded}
for a slightly more general result).
Other evidences of the 
existence of free group algebras in division rings
can be found in \cite{Cauchoncorps}, \cite{Makar-LimanovOnsubalgebrasofthe}, 
\cite{Lichtmanfreeuniversalenveloping},
\cite{SanchezfreegroupalgebraMNseries}.
Therefore it makes sense to consider the following unifying
conjecture:
\begin{enumerate}[(GA)]
\item Let $D$ be a skew field with center $Z$.
If $D$ is  finitely generated as a division ring over $Z$ and $D$ is
infinite dimensional over $Z$, then $D$ contains a noncommutative
free group $Z$-algebra.
\end{enumerate}
For more details on these and related conjectures the reader is
referred to \cite{GoncalvesShirvaniSurvey}.

A result that allowed to
assert the existence of free group algebras from the existence of a free algebra is 
\cite[Proposition~4]{LichtmanmatrixringsII}, but, unfortunately, it is wrong.
Another incorrect proof was given by the author in \cite[Lemma~7.3]{HerberaSanchez}.
We supply a counterexample to the proposition in Appendix~\ref{sec:counterexample}. 

The main aim of this work is Theorem~\ref{coro:divisionrings} which provides what can
be regarded as a correct
statement of \cite[Proposition~4]{LichtmanmatrixringsII} and a generalization  in some aspects.
Theorem~\ref{coro:divisionrings} gives sufficient conditions for the existence of a free group algebra
in a division ring endowed with a valuation with values on an ordered group. Roughly speaking,
it says that the existence of a certain free algebra in the graded
division ring induced by the valuation implies the existence of a free group algebra in the original 
division ring. Philosophically our theorem may be viewed as a result in the vein of what happens
when studying finite dimensional (over their center) division rings endowed with a valuation
\cite{TignolWadsworthbook}:   some information about the division ring
can be recovered from the graded ring, moreover this object
usually turns out to be  simpler and easier to work with than the original division ring.

\medskip

We now proceed to give a more precise description of our techniques and statement of our results.

Let $R$ be a ring, $(G,<)$ an ordered group
and $\upsilon\colon R\rightarrow G\cup\{\infty\}$ be a valuation.
For $S\subset D$ and $g\in G$ we denote
$$S_{\geq g}=\{s\in S\colon \upsilon(s)\geq g\},\quad S_{>g}=\{s\in S\colon \upsilon(s)>g\}.$$ Recall
that $$\gr_\upsilon(R)=\bigoplus_{g\in G} R_g, \textrm{ where } R_g=R_{\geq g}/R_{>g}.$$

When the value group of the valuation is a subgroup of the real numbers, our main result  is as follows.

\begin{theorem*}\label{coro:divisionrings}
Let $D$ be a division ring with prime subring $Z$. Let  $\upsilon\colon
D\rightarrow \mathbb{R}\cup\{\infty\}$ be a nontrivial valuation.
 Let
$X$ be a subset of $D$ satisfying the following
three conditions.
\begin{enumerate}[\rm (1)]
\item The map $X\rightarrow\gr_\upsilon(D)$, $x\mapsto x+D_{>\upsilon(x)}$, is injective.
\item For each $x\in X$, $\upsilon(x)>0$.
\item The $Z_0$-subalgebra of $\gr_\upsilon(D)$ generated by the set
$\{x+D_{>\upsilon(x)}\}_{x\in X}$ is the free $Z_0$-algebra
on the set $\{x+D_{>\upsilon(x)}\}_{x\in X}$, where 
$Z_0\coloneqq Z_{\geq 0}/Z_{>0}\subseteq D_0$.
\end{enumerate}
Then, for any central subfield $k$,  the
$k$-subalgebra of $D$ generated by $\{1+x,\, (1+x)^{-1}\}_{x\in
X}$ is the free group $k$-algebra on $\{1+x\}_{x\in X}$. \qed
\end{theorem*}

The main idea of the proof of this result is the same as the one sketched by
Lichtman in \cite[Proposition~4]{LichtmanmatrixringsII}. One must show
that the completion of the subring $O=\{f\in D\geq 0\}$ of $D$ with respect
to the descending chain of ideals $L_n=\{f\in D\colon \upsilon(f)\geq n\}$, $n\in\mathbb{N}$,
contains a ring of formal power series $Z\langle\langle X\rangle\rangle$
in a subset $X$ of $D$. Then the elements of $O$ corresponding to $1+x$
generate a free group $Z$-algebra by a result of Fox, see Theorem~\ref{theo:MagnusFox} below.

When trying to do the same when the value group is a general ordered group $(G,<)$, the main obstruction
is the following situation that does not happen when the ordered group is a subgroup of the real numbers.
There may be elements $g,h\in G$, $g,h>0$, such that  $g>h^n$ for all 
positive integers $n$. We say that  $g\gg h$. If neither $g\gg h$ nor $h\gg g$, we say that $g\sim h$.  
That is why we have to impose the extra condition (3) in Theorem~\ref{coro:divisionrings}
and obtain a somehow different result. 
With this condition, we are able to reduce the proof to the previous case and obtain
Theorem~\ref{coro:divisionrings}.

\begin{theorem*}\label{coro:divisionrings}
Let $D$ be a division ring with prime subring $Z$. Let $(G,<)$ be an ordered group and $\upsilon\colon
D\rightarrow G\cup\{\infty\}$ be a nontrivial valuation. Let
$X$ be a subset of $D$ satisfying the following
four conditions.
\begin{enumerate}[\rm (1)]
\item The map $X\rightarrow\gr_\upsilon(D)$, $x\mapsto x+D_{>\upsilon(x)}$, is injective.
\item For each $x\in X$, $\upsilon(x)>0$.
\item For all $x,y\in X$, $\upsilon(x)\sim\upsilon(y)$.
\item The $Z_0$-subalgebra of $\gr_\upsilon(D)$ generated by the set
$\{x+D_{>\upsilon(x)}\}_{x\in X}$ is the free $Z_0$-algebra
on the set $\{x+D_{>\upsilon(x)}\}_{x\in X}$, where 
$Z_0\coloneqq Z_{\geq 0}/Z_{>0}\subseteq D_0$.
\end{enumerate}
Then, for any central subfield $k$, the following hold true.
\begin{enumerate}[\rm (a)]
\item If there does not exists $z\in Z$ such that
$\upsilon(z)\gg\upsilon(x)$ for all $x\in X$, then the
$k$-subalgebra of $D$ generated by $\{1+x,\, (1+x)^{-1}\}_{x\in
X}$ is the free group $k$-algebra on $\{1+x\}_{x\in X}$.
\item If there exists $z\in Z$ such that $\upsilon(z)\gg
\upsilon(x)$ for all $x\in X$, then the $k$-subalgebra of $D$
generated by $\{1+zx,\, (1+zx)^{-1}\}_{x\in X}$ is the free
group $k$-algebra on the set $\{1+zx\}_{x\in X}$. \qed
\end{enumerate}
\end{theorem*}

Note that if the value group is a subgroup of the reals, then condition (3)
is trivially satisfied and (b) never occurs, obtaining in this way the foregoing result.

In Sections~\ref{sec:generalitiespseudo}, \ref{sec:valuerealnumber}, \ref{sec:valuegeneralordered}
the results are stated for rings  endowed with  filtrations that not necessarily come
from  valuations. This is because  we realized that it is necessary to state the results for 
general rings (not only division rings) to make the reduction to the real case, and because 
 it is not much harder to prove similar theorems as the ones stated for these more general class of filtrations. 

In a second part of this work, we use the theorems stated to produce new results on the
existence of free group algebras in division rings.

\medskip

Our work is structured as follows. 
In Section~\ref{sec:preliminaries}, we fix the notation and present some results on the existence of free 
(group) algebras that will be needed in later sections. In Section~\ref{sec:generalitiespseudo},
the concept of positive filtration is introduced. Roughly speaking, it is a 
descending filtration on a $Z$-algebra $R$ with values on an ordered group.
It allows us to look at the induced graded ring $\gr(R)$ as a $\gr(Z)$-ring. Then we give
some of its basic but important properties. We end this section proving that the existence
of a free object (monoid, group, algebra, group algebra)  generated by 
homogeneous elements in the graded ring, implies the existence of the same kind of
free object in the original ring.
In Section~\ref{sec:valuerealnumber}, we give sufficient conditions for the existence
of a free group algebra in an algebra endowed with a positive filtration with values in a subgroup
of the real numbers. When the value group is a general ordered group, the same is shown in
Section~\ref{sec:valuegeneralordered} using the
results of previous sections. At the beginning of this section, the needed background
on ordered groups is recalled. In the last section, we show the main result
of our work and obtain an interesting corollary. In Appendix~\ref{sec:counterexample},
we give a counterexample to \cite[Proposition~4]{LichtmanmatrixringsII}. In Appendix~\ref{sec:corrigenda},
we comment on some results where \cite[Proposition~4]{LichtmanmatrixringsII} has been
used to prove the existence of a free group algebra.

	
\section{Preliminaries}\label{sec:preliminaries}

All rings and algebras $R$ are assumed to be associative rings with
an identity element $1\in R$.

For a commutative ring $Z$, when we say that a ring $A$ is a
\emph{$Z$-algebra} we mean that \emph{$Z$ is a subring of the center
of $A$}.

Let $Y=\{X_1,\dotsc,X_n\}$ be a finite nonempty set. The \emph{free
monoid} $Y^*$ generated by $Y$ is the set of finite sequences of
elements of $Y$ including the empty sequence denoted by $1$. The
elements of $Y^*$ are called \emph{words} (on $X_1,\dotsc,X_n$). The
product in $Y^*$ is the concatenation of words and with neutral
element $1$. A word $\omega\in Y^*$ is said to be a word of \emph{length}
$r$ if $\omega(X_1,\dotsc,X_n)=X_{i_1}\dotsb X_{i_r}$ where $i_j\in\{1,\dotsc,n\}$ .
The word $1$ is the only word of length $0$.

Let $Z$ be a commutative ring. A \emph{formal series} is a function
$$Y^*\rightarrow Z, \quad \omega\mapsto a_\omega,$$
denoted as $$\sum_{\omega\in Y^*} a_\omega \omega.$$ The
\emph{$Z$-algebra of formal series on $Y$ with coefficients in $Z$}
is denoted by $Z\langle\langle Y\rangle\rangle$ or $Z\langle\langle
X_1\dotsc,X_n\rangle\rangle$. If $$f=\sum_{\omega\in
Y^*}a_\omega\omega,\, g=\sum_{\omega\in Y^*}b_\omega\omega$$ are two
formal series, then their sum and product are given by
$$f+g=\sum_{\omega\in Y^*}(a_\omega+b_\omega)\omega,\quad
fg=\sum_{\omega\in Y^*}(\sum_{uv=\omega}a_ub_v)\omega.$$ There is an
injection $Z\rightarrow Z\langle\langle Y\rangle\rangle$, $z\mapsto
z 1$.

The \emph{free $Z$-algebra on the set $Y$}, denoted by $Z\langle
Y\rangle$, is the monoid $Z$-algebra of the free monoid $Y^*$. It
can be regarded as the set of series $f=\sum\limits_{\omega\in
Y^*}a_\omega\omega\in Z\langle\langle Y\rangle\rangle$ such that the
set $\{\omega\colon a_\omega\neq 0\}$ is finite with the sum and
product induced from $Z\langle\langle Y\rangle\rangle$.

Recall that the $Z$-algebra $Z\langle\langle Y\rangle\rangle$ can
also be seen as the completion $\varprojlim Z\langle Y\rangle/J^n$
where $J$ is the ideal of $Z\langle Y\rangle$ generated by $Y$ and
$J^n$ is the $n$-th power of $J$. Thus $J^n$ is the ideal generated
by the words of length $n$.

The \emph{free group $Z$-algebra on the set $Y$} is the group
$Z$-algebra $Z[G]$ where $G$ is the free group on the set $Y$.

Let $a_1,\dotsc,a_n$ be a sequence of elements of a ring (algebra,
group, monoid,\dots) $R$. Let $\omega\in Y^*$,
$\omega(X_1,\dotsc,X_n)=X_{i_1}\dotsb X_{i_r}$ where
$i_j\in\{1,\dotsc,n\}$. By $\omega(a_1,\dotsc,a_n)$, we mean
$$\omega(a_1,\dotsc,a_n)=a_{i_1}\dotsb a_{i_r}\in R.$$
That is, the evaluation of the word $\omega(X_1,\dotsc,X_n)$ on the
sequence $a_1,\dotsc,a_n$.

Sometimes, we will deal with free algebras and free group algebras where 
the ring of coefficients need not be commutative.
Let $S$ be a ring. By the \emph{free $S$-ring on $Y$}, we mean the
monoid $S$-ring os the free monoid $Y^*$. In other words  
the elements of $S$ are polynomials in the noncommuting variables of $Y$
with coefficients from $S$. Here, the coefficients are supposed to commute
with each $X_i$. Analaogously, one can define the \emph{free group $S$-ring on $Y$},
as the group ring $S[G]$ where $G$ is the free group on $Y$.

\bigskip

The main tool that we will use to obtain free group algebras inside
rings is the next result.

\begin{theo}\label{theo:MagnusFox}
Let $Y=\{X_1,\dotsc,X_n\}$ be a finite set,  $G$ be the free group
on $Y$ and $Z$ be a commutative ring.  Then $Z[G]$, the free group
$Z$-algebra on $Y$, embeds in $Z\langle\langle Y\rangle\rangle$ via
the $Z$-algebra homomorphism $\Psi\colon Z[G]\rightarrow
Z\langle\langle Y\rangle\rangle$ determined by $X_i\mapsto 1+X_i$,
$X_i^{-1}\mapsto (1+X_i)^{-1}.$ \qed
\end{theo}

Theorem~\ref{theo:MagnusFox} goes back to R. Fox \cite[Theorem~4.3]{Foxembedding} who
stated the result for $Z=\mathbb{Z}$, but the proof of \cite{Foxembedding} works for any ring $Z$ (even not  commutative)
as noted in \cite[2.12]{AraDicks}. The proof of Fox's result was reviewed in \cite[Section~2]{AraDicks}.
Other proofs can be found in \cite{AraDicksII}, \cite{LichtmanmatrixringsII}.

\medskip

Suppose that $D$ is a division ring with prime subfield $P$. Let $C$
be a central subfield of $D$. The next result tells us that finding
free (group) $P$-algebras inside $D$ is equivalent to finding free
(group) $C$-algebras inside $D$. It will be used throughout the
paper.

\begin{lem}\label{lem:MalcolmsonMakarLimanov}
Let $D$ be a domain with prime subfield $P$ and 
suppose that $C$ is a  subfield of $D$ contained in the center of $D$. 
Let $M$ be a free submonoid
(subgroup) of the multiplicative group $D\setminus\{0\}$. Then the
algebra generated by $P$ and $M$ is the monoid (group) algebra
$P[M]$ if, and only if, the algebra generated by $C$ and $M$ is the
monoid (group) $C$-algebra $C[M]$. \qed
\end{lem}

Lemma~\ref{lem:MalcolmsonMakarLimanov} was proved  for $M$ a free monoid in
\cite[Lemma~1]{MakarMalcolmson91}. The fact that a similar proof works
for   free groups $M$ was noted in
\cite{GoncalvesShirvani}.

\section{Generalities on positive filtrations}\label{sec:generalitiespseudo}

Throughout this work, although we deal with not necessarily abelian groups
$G$, the operation is denoted additively unless otherwise stated. Thus, in general, $g+h\neq
h+g$ for $h,g\in G$. The identity element of $G$ is denoted by $0$
and the inverse of an element $g\in G$ by $-g$.

A \emph{strict ordering} on a set $S$ is a binary relation $<$ which is transitive and such that
$s_1<s_2$ and $s_2<s_1$ cannot both hold for elements $s_1,s_2\in S$. It is a \emph{strict total
ordering} if for every $s_1,s_2\in S$ exactly one of $s_1<s_2$, $s_2<s_1$ or $s_1=s_2$ holds.

A group $G$ is called \emph{orderable} if its elements can be given a strict total ordering $<$
which is left and right invariant. That is, $g_1<g_2$ implies that $g_1+h<g_2+h$ and $h+g_1<h+g_2$ for all $g_1,g_2, h\in G$.
We call the pair $(G,<)$ an \emph{ordered group}.
Clearly, any additive subgroup of the real numbers is orderable. More generally, torsion-free abelian groups,
torsion-free nilpotent groups and residually torsion-free nilpotent groups are orderable \cite{Fuchs}.

Let $(G,<)$ be an ordered group. A symbol $\infty$ is adjoined to
$G$ and in $G_\infty=G\cup\{\infty\}$ the operation and order are
extended in such a way that $g+\infty=\infty+g=\infty+\infty=\infty$
and $g<\infty$, for all $g\in G$.

Let $R$ be a ring and let $(G,<)$ be an ordered group.  A
\emph{valuation} on $R$ with values in $G$ is a map $\upsilon\colon
R\rightarrow G_\infty$ that satisfies the following conditions for
all $x,y\in R$,
\begin{enumerate}[\hspace{0.5cm}(V.1)]
\item $\upsilon(x)=\infty$ if and only if $x=0$;
\item $\upsilon(x+y)\geq \min\{\upsilon(x),\upsilon(y)\}$;
\item $\upsilon(xy)= \upsilon(x)+\upsilon(y)$.
\end{enumerate}
If $\upsilon(x)=0$ for all $x\in R\setminus\{0\}$, we say that
$\upsilon$ is the \emph{trivial valuation} on $R$.

It is well known that $\upsilon(1)=\upsilon(-1)=0,$ $\upsilon(x)=\upsilon(-x)$ for any $x\in R$
and $\upsilon(x+y)=\min\{\upsilon(x), \upsilon(y)\}$ provided $x,y\in R$ with $\upsilon(x)\neq \upsilon(y)$.  

Notice that if $R$ is a nonzero ring, then it is a domain by condition (V.3).
Thus the characteristic of $R$ is either zero or a prime number $p$.
Let $Z$ be the \emph{prime subring} of $R$, that is, the subring of
$R$ generated by $1$. If the characteristic of $R$ is  a prime number $p$, then
$Z$ can be identified with $\mathbb{Z}/p\mathbb{Z}$, and the
restriction of $\upsilon$ to $Z$ is the trivial valuation on $Z$.
When the characteristic of $R$ is zero, $Z$ can be identified with
the ring of integers $\mathbb{Z}$ and  the restriction of $\upsilon$ to
$Z$ is explained in Lemma~\ref{lem:valuationintegers}.

\begin{lem}\label{lem:valuationintegers}
Let $(G,<)$ be an ordered group, and $\upsilon\colon
\mathbb{Z}\rightarrow G_\infty$ be a nontrivial valuation.
Then there exist a prime $p\in\mathbb{Z}$ and $g\in G$, $g>0$, such
that \begin{equation}\label{eq:valuationintegers} \upsilon
(z)=rg, \textrm{ for each } z\in\mathbb{Z}\setminus\{0\}, \end{equation} where $z$ is uniquely decomposed as 
$z=p^rb$ with  $r\in\mathbb{N}$ and $b$  an integer coprime with $p$.

Conversely, for each prime $p\in\mathbb{Z}$ and $g\in G$, $g>0$, there
exists a valuation $\upsilon\colon\mathbb{Z}\rightarrow G_\infty$
defined as in \eqref{eq:valuationintegers}.
\end{lem}

\begin{proof}
Clearly, if $\upsilon(q)=0$ for all prime numbers $q\in\mathbb{Z}$,
then $\upsilon(z)=0$ for all $z\in\mathbb{Z}$.

It is well known that any valuation on $\mathbb{Z}$ can be uniquely
extended to a valuation on $\mathbb{Q}$.  Since
$$0=\upsilon(1)=\upsilon(\frac{q}{q})=\upsilon(\frac{1}{q}+\dotsb+\frac{1}{q})
\geq \upsilon(\frac{1}{q}),$$ then $\upsilon(q)>0$ for any prime $q$
such that $\upsilon(q)\neq 0$. Hence $\upsilon(z)\geq 0$ for any
$z\in\mathbb{Z}$.

Let $p$ and $q$ be two different prime numbers and suppose that $\upsilon(p)\neq 0$. Then there exist
integers $a,b$ such that $ap+bq=1$. Since
$$0=\upsilon(1)=\upsilon(ap+bq)\geq \min\{\upsilon(ap),\upsilon(bq)\}\geq\min\{\upsilon(p),\upsilon(q)\},$$
then $\upsilon(q)=0$. Hence there exists at
most one prime number $p$ such that $\upsilon(p)>0$.

Given an integer $z$, there exists a decomposition of $z$ of the
form $z=p^rb$ with $r\geq 0$ and $b\in\mathbb{Z}$ coprime with $p$.
Hence, if $\upsilon(p)=g$, $\upsilon(z)=rg$, as desired.

Conversely, if $p$ is a prime number, and $g\in G$, $g>0$, it is
easy to check that $\eqref{eq:valuationintegers}$ defines a
valuation.
\end{proof}

Let $Z$ be a commutative ring and $R$ be a $Z$-algebra. Let $(G,<)$
be an ordered group. A \emph{positive filtration} (over $Z$) with
values in $G$ is a map $\upsilon\colon R\rightarrow G_\infty$ that
satisfies the following conditions for all $x,y\in R$, $a\in Z$
\begin{enumerate}[\hspace{0.5cm}(PF.1)]
\item $\upsilon(x)=\infty$ if and only if $x=0$;
\item $\upsilon(x+y)\geq \min\{\upsilon(x),\upsilon(y)\}$;
\item $\upsilon(xy)\geq \upsilon(x)+\upsilon(y)$;
\item $\upsilon(a)\geq 0$.
\end{enumerate}

\medskip

The following lemmas present some trivial properties about
positive filtrations that will be used throughout without any further reference.

\begin{lem}\label{lem:positivefiltrations}
Let $Z$ be a commutative ring and $R$ be a $Z$-algebra. Let $(G,<)$
be an ordered group and $\upsilon \colon R\rightarrow G_\infty$ be a positive filtration.
The following assertions hold true.
\begin{enumerate}[\rm(1)]
	\item $\upsilon(1)=\upsilon(-1)=0$.
	\item $\upsilon(x)=\upsilon(-x)$ for all $x\in R$.
	\item $\upsilon(x+y)=\min\{\upsilon(x),\upsilon(y)\}$ provided $\upsilon(x)\neq\upsilon(y)$.
	\end{enumerate}
\end{lem}
\begin{proof}
(1) Notice that $1,-1\in Z$, thus $\upsilon(1),\upsilon(-1)\geq 0$. Now
\begin{equation*}
\begin{split}
\upsilon(1)=\upsilon(1\cdot 1) &\geq \upsilon(1)+\upsilon(1) \\
\upsilon(1)=\upsilon((-1)(-1)) & \geq  \upsilon(-1)+\upsilon(-1).
\end{split}
\end{equation*}
The first inequality shows that $\upsilon(1)\leq 0$, and therefore $\upsilon(1)=0$.
Using this, the second inequality shows that $\upsilon(-1)\leq 0$, and therefore 
$\upsilon(-1)=0$.

(2) Now $\upsilon(-x)=\upsilon((-1)x)\geq \upsilon(-1)+\upsilon(x)=\upsilon(x)$. Similarly
$\upsilon(x)\geq \upsilon(-x)$. Hence $\upsilon(x)=\upsilon(-x)$.

(3) The same proof as for valuations works.
\end{proof}

\begin{lem}
A valuation $\upsilon\colon R\rightarrow G_\infty$ is a positive filtration (over
the prime subring $Z$ of $R$). 
\end{lem}

\begin{proof}
Clearly (PF.1), (PF.2) and (PF.3) hold    . 
If $\upsilon(a)$=0 for all nonzero $a\in Z$, then
(PF.4) trivially satisfied. 
If the characteristic of $R$ is zero and $\upsilon$
is not the trivial valuation when restricted to $Z$, Lemma~\ref{lem:valuationintegers} 
implies (PF.4). 
\end{proof}

\medskip

The following procedure to obtain the induced graded ring from
a positive filtration is well known. 
 Let $Z$ be a commutative ring and $R$ be a $Z$-algebra. Let
$(G,<)$ be an ordered group and $\upsilon\colon R\rightarrow
G_\infty$ be a positive filtration. For each $g\in G$, define
$$R_{\geq g}=\{x\in R\colon \upsilon (x)\geq g\}, \quad
R_{>g}=\{x\in R\colon \upsilon(x)>g\}.$$ 
By (PF.2) and Lemma~\ref{lem:positivefiltrations}, $R_{\geq g}$ and $R_{>g}$ are abelian groups.
Define
$$R_g=R_{\geq g}/R_{>g}.$$  The fact that $G$ is an ordered group and (PF.3) imply that 
$$R_{\geq g}R_{\geq h}\subseteq R_{\geq g+h},\ R_{>g}R_{>h}\subseteq
R_{>g+h},\
 R_{>h}R_{\geq g}\subseteq R_{>g+h},\ R_{\geq g}R_{>h}\subseteq
R_{>g+h}$$ for any
$g,h\in G$. Thus a multiplication can be defined by
\begin{equation}\label{eq:filtrationmultiplication}
R_g\times R_h\longrightarrow R_{g+h},\quad
(x+R_{>g})(y+R_{>h})=(xy)+R_{>g+h}.
\end{equation}
The \emph{associated graded ring} of $\upsilon$ on $R$ is defined to
be
$$\gr_\upsilon(R)=\bigoplus_{g\in G} R_g.$$ The addition on $\gr_\upsilon(R)$ arises
from the addition on each component $R_g$. The multiplication is
defined by extending the multiplication
\eqref{eq:filtrationmultiplication} on the components bilinearly to
all $\gr_\upsilon(R)$. Note that $1+R_{>0}$ is the identity element
of $\gr_\upsilon(R)$.

Consider the restriction of $\upsilon$ to $Z$. Then $\gr_\upsilon(Z)$
can be regarded as a  subring of $\gr_\upsilon(R)$. Since, in general,
$\upsilon(a)+\upsilon(x)\neq \upsilon(x)+\upsilon(a)$ for $a\in Z$, $x\in R$, we cannot
consider $\gr_\upsilon(R)$ as a $\gr_\upsilon(Z)$-algebra. If the group $G$ is abelian,
then we can consider $\gr_\upsilon(R)$ as a $\gr_\upsilon(Z)$-algebra. 
On the other hand, whether $G$ is abelian or not,
$\gr_\upsilon(R)$ is always $Z_0$-algebra, where
$Z_0=Z_{\geq 0}/Z_{>0}$.

In the next lemma, Lemma~\ref{lem:valuationsum}(2) can be found in \cite[Proposition~2.6.1]{Cohnskew}.
The other statements are straightforward. This result will be used throughout, sometimes without further reference.

\begin{lem}\label{lem:valuationsum}
Let $Z$ be a commutative ring and $R$ be a $Z$-algebra. Let $(G,<)$
be an ordered group and $\upsilon\colon R\rightarrow G_\infty$ be a
positive filtration.  Let $f_1,f_2,\dotsc,f_r\in R$.
\begin{enumerate}[\rm(1)]
\item  $(f_1+R_{>\upsilon(f_1)})(f_2+R_{>\upsilon(f_2)})\dotsb (f_r+R_{>\upsilon(f_r)})\neq 0$
if, and only if, $\upsilon(f_1f_2\dotsb f_r)=\upsilon(f_1)+\upsilon(f_2)+\dotsb+\upsilon(f_r)$.

\item $\gr_\upsilon(R)$ is a domain, if and only if, $\upsilon$ is a valuation.

\item Suppose that
$\upsilon(f_1)=\dotsb=\upsilon(f_r)=g\in G$. The following statements are
equivalent
\begin{enumerate}[\rm(a)]
\item $\upsilon(f_1+\dotsb+f_r)=g$,
\item $(f_1+R_{>g})+\dotsb+(f_r+R_{>g})=(f_1+\dotsb+f_r)+R_{>g}\neq
0$ as an element of $R_g$. \qed
\end{enumerate}
\end{enumerate}
\end{lem}

We end this section with the following result. We prove that the existence of a
free object generated by homogeneous elements in the graded ring   
implies the existence of a free object of the same kind inside the ring $R$.

\begin{prop}\label{prop:freeobjecthomogeneous}
Let $Z$ be a commutative ring and $R$ be a $Z$-algebra. Let $(G,<)$ be an ordered group and
$\upsilon\colon R\rightarrow G_\infty$ be a positive filtration. Let $X$
be a subset of elements of $R$ such that
the map $X\rightarrow \gr_\upsilon(R)$, $x\mapsto x+R_{>\upsilon(x)}$,
is injective. The following statements hold true.
\begin{enumerate}[\rm(1)]
\item If the multiplicative submonoid of $\gr_\upsilon(R)$ generated by $\{x+R_{>\upsilon(x)}\}_{x\in
X}$   is the free monoid on $\{x+R_{>\upsilon(x)}\}_{x\in X}$,
then the multiplicative submonoid of $R$ generated by $X$ is the free
submonoid on $X$.
\item Suppose that:
\begin{enumerate}[\rm(a)]
\item The elements of $X$ are invertible in $R$ and $\upsilon(x^{-1})=-\upsilon(x)$
for all $x\in X$;
\item The multiplicative monoid of $\gr_\upsilon(R)$ generated by
$\{x+R_{>\upsilon(x)},\, x^{-1}+R_{>\upsilon(x^{-1})}\}_{x\in X}$ is the free group on
$\{x+R_{>\upsilon(x)}\}_{x\in X}$.
\end{enumerate}
Then the multiplicative subgroup of $R$
generated by $X$ is the free group on $X$.
\item If the subring of $\gr_\upsilon(R)$ generated by $\gr_\upsilon(Z)$ and $\{x+R_{>\upsilon(x)}\}_{x\in
X}$   is the free $\gr_\upsilon(Z)$-ring on $\{x+R_{>\upsilon(x)}\}_{x\in X}$,
then the  $Z$-subalgebra of $R$ generated by $X$ is the free $Z$-algebra on
$X$.
\item Suppose that 
\begin{enumerate}[\rm(a)]
\item The elements of $X$ are invertible in $R$ and $\upsilon(x^{-1})=-\upsilon(x)$
for all $x\in X$;
\item The subring of $\gr_\upsilon(R)$ generated by $\gr_\upsilon(Z)$ and
$\{x+R_{>\upsilon(x)},\, x^{-1}+R_{>\upsilon(x^{-1})}\}_{x\in X}$
is the free group $\gr_\upsilon(Z)$-ring on $\{x+R_{>\upsilon(x)}\}_{x\in
X}$.
\end{enumerate}
Then the $Z$-subalgebra of $R$ generated by $\{x,\,
x^{-1}\}_{x\in X}$ is the free group $Z$-algebra on $X$.
\end{enumerate}
If, moreover, $\upsilon$ is a valuation, one can change {\rm(3)} and {\rm(4)}, respectively, for
\begin{enumerate}
\item[\rm(3')] If the $Z_0$-subalgebra of $\gr_\upsilon(R)$ generated by $\{x+R_{>\upsilon(x)}\}_{x\in
X}$   is the free $Z_0$-algebra on $\{x+R_{>\upsilon(x)}\}_{x\in X}$,
then the  $Z$-subalgebra of $R$ generated by $X$ is the free $Z$-algebra on
$X$.
\item[\rm(4')] Suppose that:
\begin{enumerate}[\rm(a)]
\item The elements of $X$ are invertible.

\item The $Z_0$-subalgebra of $\gr_\upsilon(R)$ generated by
$\{x+R_{>\upsilon(x)},\, x^{-1}+R_{>\upsilon(x^{-1})}\}_{x\in X}$
is the free group $Z_0$-algebra on $\{x+R_{>\upsilon(x)}\}_{x\in
X}$.
\end{enumerate}
\end{enumerate}
\end{prop}

\begin{proof}
We prove (2) and (4). The others can be shown in the same way.

The fact that
$\upsilon(x^{-1})=-\upsilon(x)$  implies that 
$x+R_{>\upsilon(x)}$ is invertible in $\gr_\upsilon(R)$
with inverse $x^{-1}+R_{>\upsilon(x^{-1})}$ for each $x\in X$.

For any $f\in R$, we
write $\tilde{f}=f+R_{>\upsilon(f)}\in \gr_\upsilon(R)$.

We may suppose that $X=\{x_1,\dotsc,x_n\}$. 
For any nontrivial reduced group word $\omega$ on
the variables $\{X_1,\dotsc,X_n\}$, by hypothesis (b), one has
$\omega(\tilde{x}_1,\dotsc,\tilde{x}_n)=\widetilde{\omega(x_1,\dotsc,
x_n)}\neq \tilde{1}$. Thus $\omega(x_1,\dotsc,x_n)\neq 1$. Hence the
multiplicative group generated by $X$ is the free group on $X$ and (2)
is proved.

Let now $\omega_1,\dotsc\omega_r$ be different reduced group words
on $\{X_1,\dotsc,X_n\}$, and let $a_1,\dotsc,a_r\in Z\setminus\{0\}$.
Suppose that $\upsilon(x_i)=g_i$,
$i=1,\dotsc,n$. By Lemma~\ref{lem:valuationsum}(1)
$$\upsilon(a_j\omega_j(x_1,\dotsc,x_n))=\upsilon(a_j)+\omega_j(g_1,\dotsc,g_n),\ 
j=1,\dotsc,r,$$
where $\omega_j(g_1,\dotsc,g_n)$ is expressed in additive notation.
  We may suppose that 
$$\upsilon(a_1\omega_1(x_1,\dotsc,x_n))
\leq
\dotsb \leq \upsilon(a_r\omega_r(x_1,\dotsc,x_n)).$$ Let $t$ be the least
positive integer such that 
\begin{multline*}
\upsilon(a_1\omega_1(x_1,\dotsc,x_n))=\dotsb
=\upsilon(a_t\omega_t(x_1,\dotsc,x_n))<\upsilon(a_{t+1}\omega_{t+1}(x_1,\dotsc,x_n)) \leq\\
\dotsb\leq \upsilon(a_r\omega_r(x_1,\dotsc,x_n)).
\end{multline*}
 By assumption (4)(b), 
$$0\neq \sum_{j=1}^t \widetilde{a}_j\omega_j(\tilde{x}_1,\dotsc,\tilde{x}_n)\in
R_{\upsilon(a_1\omega_1(x_1,\dotsc,x_n))}.$$
By Lemma~\ref{lem:valuationsum}(3),
$$\upsilon\left(\sum_{j=1}^ta_j\omega_j(x_1,\dotsc,x_n)\right)=\upsilon(a_1\omega_1(x_1,\dotsc,x_n)).$$
Thus $\upsilon \Big(\sum\limits_{j=1}^ra_j\omega_j(x_1,\dotsc,x_n)\Big)=\upsilon(a_1\omega_1(x_1,\dotsc,x_n))\neq \infty$ and
therefore \linebreak $\sum\limits_{j=1}^ra_j\omega_j(x_1,\dotsc,x_n)\neq 0$, as
desired.

Suppose now that $\upsilon$ is a valuation. Hence
$\upsilon(x^{-1})=-\upsilon(x)$
for all $x\in X$ and $\gr_\upsilon(R)$ is a domain. 
We now show that (4')(b) is equivalent to (4)(b).
First note that $\upsilon(a)+\upsilon(f)=\upsilon(af)=\upsilon(fa)=\upsilon(f)+\upsilon(a)$ 
for all $a\in Z$, $f\in R$. Hence $\gr_\upsilon(Z)$ is commutative and
$\gr_\upsilon(R)$ is a $\gr_\upsilon(Z)$-algebra.
Consider the central multiplicative subset of $\gr_\upsilon(R)$
$S=\gr_\upsilon(Z)\setminus \{0\}$, and
localize $\gr_\upsilon(R)$ at $S$. Then $\gr_\upsilon(R)$ embeds in the domain $S^{-1}\gr_\upsilon(R)$. Let $Q_0$
be the subfield of $C=S^{-1}\gr_\upsilon(Z)$ generated by $Z_0$. By Lemma~\ref{lem:MalcolmsonMakarLimanov}, the
$C$-subalgebra of $S^{-1}\gr_v(R)$ generated by a set $Y$ is the free $C$-subalgebra on $Y$ if, and only if,
the $Q_0$-subalgebra of $S^{-1}\gr_v(R)$ generated by  $Y$ is the free $Q_0$-subalgebra on $Y$. Now observe that these 
subalgebras are free over $C$ (respectively $Q_0$) if, and only if, they are free over $\gr_\upsilon(Z)$ (respectively $Z_0$). 
\end{proof}


\section{When the value group is a subgroup of the real numbers}\label{sec:valuerealnumber}

Let $Y=\{X_1,\dotsc,X_n\}$ be
a set of $n$ noncommutative variables.
Fix $n$ (not necessarily different) positive real numbers
$p_1,p_2,\dotsc,p_n$. For each  $\omega\in Y^*$,
$$\omega=\omega(X_1,X_2,\dotsc,X_n)=X_{i_1}X_{i_2}\dotsb X_{i_s},$$
where $i_j\in\{1,\dotsc,n\}$,  we define the \emph{weighted degree}
of $w$ as
$$\wed(\omega)=\sum_{j=1}^s p_{i_j}.$$
Observe that $\wed(1)=0$ and
 $\wed(\omega\omega')=\wed(\omega)+\wed(\omega')$ for all $\omega,\omega'\in Y^*$.

Let $Z$ be a commutative ring endowed with a positive filtration
$\eta\colon Z\rightarrow \mathbb{R}_\infty$  such that
$\eta(z)\geq 0$ for all $z\in Z$. 
Define the  map $\wed_\eta\colon Z\langle Y\rangle \rightarrow
\mathbb{R}_\infty$ as follows. Let $f\in Z\langle Y\rangle$,
$f=\sum\limits_{\omega\in Y^*} z_\omega \omega$, then
\begin{equation*}
\wed_\eta(f) =   \min\{\eta(z_\omega)+\wed(\omega)\colon \omega\in Y^*\}.
\end{equation*}

\begin{lem}\label{lem:completion}
Let $Z$ be a commutative ring endowed with a positive filtration
$\eta\colon Z\rightarrow \mathbb{R}_\infty$  such that
$\eta(z)\geq 0$ for all $z\in Z$. Let $Y=\{X_1,\dotsc,X_n\}$ be
a set of $n$ noncommutative variables. Fix $n$ (not necessarily
different) positive real numbers $p_1,p_2,\dotsc,p_n$. The following statements
hold true.
\begin{enumerate}[\rm (1)]
\item The map $\wed_\eta\colon Z\langle Y\rangle\rightarrow
\mathbb{R}_\infty$ is a positive filtration (over $Z$).
\item If $\eta\colon Z\rightarrow \mathbb{R}_\infty$ is a valuation, then
$\wed_\eta\colon Z\langle Y\rangle\rightarrow
\mathbb{R}_\infty$ is a valuation.
\item For each positive integer $d$, let $L_d=\{f\in Z\langle Y\rangle\colon \wed_\eta(f)\geq
d\}$. Then $Z\langle\langle Y\rangle\rangle \hookrightarrow
\varprojlim Z\langle Y\rangle/L_d$.
\end{enumerate}
\end{lem}

\begin{proof}
The proof is analogous to \cite[p.87]{Cohnskew}.

\noindent (1) 
Clearly $\wed_\eta(f)=\infty$ if, and only if, $f=0$. 

Let $f=\sum\limits_{\omega\in Y^*} z_\omega \omega,\, 
f'=\sum\limits_{\omega\in Y^*}z_\omega'\omega\in Z\langle Y\rangle$.

By
definition $f+f'=\sum\limits_{\omega\in Y^*}(z_\omega+z_\omega')\omega$. 
For each $\omega\in Y^*$, 
\begin{eqnarray*}
\wed_\eta((z_\omega+z_\omega')\omega) & = & \eta(z_\omega+z_\omega')+\wed(\omega) \\ & \geq & 
\min\{\eta(z_\omega),\eta(z_\omega')\}+\wed(\omega) \\ 
& = &\min\{\wed_\eta(z_\omega\omega), 
\wed_\eta(z_\omega'\omega)\}.
\end{eqnarray*}
Hence $\wed_\eta(f+f')\geq \min\{\wed_\eta(f),\wed_\eta(f')\}$.

Now $ff'=\sum\limits_{\omega\in Y^*}\Big(\sum\limits_{\alpha\beta=\omega}z_\alpha z_\beta'\Big)\omega$. 
Respectively, let $\alpha_0,\beta_0\in Y^*$ be words of least weighted degree among the words
$\alpha,\beta\in Y^*$ such that 
$\wed_\eta(f)=\eta(z_{\alpha})+\wed(\alpha)$  and 
$\wed_\eta(f')=\eta(z_{\beta}')+\wed(\beta)$. 
By definition, $$\wed_\eta(ff')=\min\Big\{\eta\Big(\sum_{\alpha\beta=\omega}z_\alpha z_\beta' \Big)+\wed(\omega)\colon \omega\in Y^*  \Big\}$$
Observe that, for all $\alpha,\beta\in Y^*$, $$\eta(z_\alpha z_\beta')+\wed(\alpha\beta)\geq 
\eta(z_{\alpha_0})+\wed(\alpha_0) +\eta(z_{\beta_0}')+\wed(\beta_0)=
\wed_\eta(f)+\wed_\eta(f').$$
Thus 
\begin{equation}\label{eq:wdispositivefiltration}
\wed_\eta(ff')\geq \wed_\eta(f)+\wed_\eta(f'),
\end{equation}
and (1) is proved.

(2) Now we must prove that equality in \eqref{eq:wdispositivefiltration} holds.
For that, it is enough to find $\omega_0\in Y^*$ for which $\eta \Big(\sum\limits_{\alpha\beta
=\omega_0}z_\alpha z_\beta'\Big)+\wed(\omega_0)=\wed_\eta(f)+\wed_\eta(f')$.

Let $\omega_0=\alpha_0\beta_0$. Consider the set 
$$A=\{(\alpha,\beta)\in Y^*\times Y^*\colon \alpha\beta=\omega_0\}\setminus \{(\alpha_0,\beta_0)\}.$$
For each $(\alpha,\beta)\in A$, either there exists $\alpha_1\in Y^*$, $\alpha_1\neq 1$,
such that $\alpha_0=\alpha\alpha_1$ or there exists $\beta_1\in Y^*$, $\beta_1\neq 1$, such that
$\beta_0=\beta_1\beta$. If $\alpha_0=\alpha\alpha_1$,  $\wed(\alpha)<\wed(\alpha_0)$, and
if $\beta_0=\beta_1\beta$, $\wed(\beta)<\wed(\beta_0)$. By the choice of $\alpha_0,\beta_0$,
we get that for each $(\alpha,\beta)\in A$ either $\eta(z_{\alpha_0})+\wed(\alpha_0)<
\eta(z_\alpha)+\wed(\alpha)$ or $\eta(z_{\beta_0}')+\wed(\beta_0)<
\eta(z_\beta')+\wed(\beta)$. Hence
\begin{eqnarray*}
\eta(z_{\alpha_0}z_{\beta_0}')+\wed(\omega_0) & = & 
\eta(z_{\alpha_0})+ \wed(\alpha_0) +\eta(z_{\beta_0}')+\wed({\beta_0}) \\
 & < & \eta(z_\alpha)+\wed(\alpha) + \eta(z_{\beta}')+\wed(\beta) \\
& = & \eta(z_\alpha z_\beta')+\wed(\omega).
\end{eqnarray*}
Therefore $\eta(z_{\alpha_0}z_{\beta_0}')<\eta(z_\alpha z_\beta ')$ for all
$(\alpha,\beta)\in A$. By the properties of valuations, $\eta\Big(\sum\limits_{\alpha\beta=\omega_0}z_\alpha z_\beta' \Big)
=\eta(z_{\alpha_0})+\eta(z_{\beta_0}')$, which implies the result.

(3) Let $J=\langle Y\rangle$ be the ideal of $Z\langle Y\rangle$ generated by $Y$.  Define
$u=\min\{p_1,\dotsc,p_n\}$. Then, for each natural number $m$,
$J^m\subseteq L_{\left[um\right]}$, where $\left[ a \right]$ denotes
the greatest integer smaller or equal to the real number $a$. This
induces natural ring homomorphisms
$$Z\langle
Y\rangle/J^m\stackrel{\phi_m}{\longrightarrow} Z\langle
Y\rangle/L_{[um]},$$ such that the following diagram is commutative
for $m_1<m_2$
$$\xymatrix{Z\langle
Y\rangle/J^{m_2}\ar[r]^{\phi_{m_2}}\ar[d]
 & Z\langle
Y\rangle/L_{[um_2]}\ar[d]\\ Z\langle
Y\rangle/J^{m_1}\ar[r]^{\phi_{m_1}} & Z\langle Y\rangle/L_{[um_1]}
}$$  Thus we obtain a ring homomorphism between the completions
$$Z\langle\langle Y\rangle\rangle=\varprojlim_m Z\langle
Y\rangle/J^m \stackrel{\phi}{\longrightarrow} \varprojlim_d Z\langle
Y\rangle/L_d.$$ Let a nonzero series $f=\sum\limits_{\omega\in
Y^*}z_\omega \omega\in Z\langle\langle Y\rangle\rangle$ be given. For each
positive integer $d$,
\begin{equation}\label{eq:finitesum}
\sum_{\wed_\eta(z_\omega\omega)< d}z_\omega\omega
\end{equation} is a finite sum. Choose $d$ such that \eqref{eq:finitesum}
is not zero in $Z\langle Y\rangle$. Let now $m$ be such that $[um]>d$.
Then $\phi(f)\neq 0$ because the projection
of $\phi(f)$ in $Z\langle Y\rangle/L_{[um]}$ is not zero. Hence $\phi$ is an injective ring homomorphism.
\end{proof}

\begin{theo}\label{theo:freegroupalgebrafiltrationreals}
Let $Z$ be a commutative ring and $R$ be a $Z$-algebra. Let
$\upsilon\colon R\rightarrow \mathbb{R}_\infty$ be a nontrivial
positive filtration. Let $X$ be a subset of elements of $R$ satisfying the following
conditions
\begin{enumerate}[\rm (1)]
\item The map $X\rightarrow\gr_\upsilon(R)$, $x\mapsto x+R_{>\upsilon(x)}$, is injective.
\item For each $x\in X$, $\upsilon(x)>0$.
\item The $\gr_\upsilon(Z)$-subalgebra of $\gr_\upsilon(R)$ generated by the set
$\{x+R_{>\upsilon(x)}\}_{x\in X}$ is the free $\gr_\upsilon(Z)$-algebra
on the set $\{x+R_{>\upsilon(x)}\}_{x\in X}$.
\item For each $x\in X$, $1+x$ is invertible in $R$ and $\upsilon((1+x)^{-1})=0$.
\end{enumerate}
Then the $Z$-subalgebra of $R$ generated by $\{1+x,\, (1+x)^{-1}
\}_{x\in X}$ is the free group $Z$-algebra on the set
$\{1+x\}_{x\in X}$.
\end{theo}

\begin{proof}
Observe that it is enough to show the result for $X$ a finite set.
Thus suppose that $X=\{x_1,\dotsc,x_n\}$. Fix the positive real
numbers $p_i=\upsilon(x_i)$, $i=1,\dotsc,n$.

Let $O=\{x\in R\colon \upsilon(x)\geq 0\}$. For each $d\geq 1$, let
$K_d=\{x\in R\colon \upsilon(x)\geq d\}$. Notice that $K_d$ is an
ideal of $O$ for each $d$.

Let $Y=\{X_1,\dotsc,X_n\}$ be a set of $n$ noncommuting variables
and consider the free $Z$-algebra $Z\langle Y\rangle$. Let
$f=\sum\limits_{\omega\in Y^*} z_\omega\omega\in Z\langle Y\rangle$
be a nonzero polynomial.
Let $t=\min\{\upsilon(z_\omega\omega(x_1,\dotsc,x_n))\colon
\omega\in Y^*\}$. Let $A=\{\omega\in Y^*\colon
\upsilon(z_\omega\omega(x_1,\dotsc,x_n))=t\}$. The set $A$ is finite
and not empty. Suppose that $A=\{\omega_1,\dots,\omega_r\}$. We can
write
$$f(x_1,\dotsc,x_n)=\sum_{j=1}^r z_{\omega_j}\omega_j(x_1,\dotsc,x_n)+
\sum_{\upsilon(z_w\omega(x_1,\dotsc,x_n))>t}z_\omega\omega(x_1,\dotsc,x_n).$$
By condition (3), the element
$\sum\limits_{j=1}^rz_{\omega_j}\omega_j(x_1,\dotsc,x_n)+R_{>t}\in
R_t$ is not zero. By Lemma~\ref{lem:valuationsum}(3),
$\upsilon\Big(\sum\limits_{j=1}^rz_{\omega_j}\omega_j(x_1,\dotsc,x_n)
\Big)=t$. Therefore $\upsilon(f(x_1,\dotsc,x_n))=t$. In particular,
this implies that $Z\langle Y\rangle\hookrightarrow O$ via $f\mapsto
f(x_1,\dotsc,x_n)$.

Let $\eta$ be the restriction of $\upsilon$ to $Z$. 
Note that for each $\omega\in Y^*$ and $z\in Z$,
$\upsilon(z\omega)=\eta(z)+\wed(w)$, where $\wed$ is as defined
before Lemma~\ref{lem:completion} and associated to $p_1,\dotsc,p_n$. The
foregoing implies that $$\upsilon(f(x_1,\dotsc,x_n))=\wed_\eta(f)
\textrm{ for all } f\in Z\langle Y\rangle,$$ 
where $\wed_\eta$ is the valuation studied in Lemma~\ref{lem:completion}.
Setting $L_d=\{f\in
Z\langle Y\rangle\colon \wed_\eta(f)\geq d\}$ for each integer $d\geq 1$,
we obtain the embedding $$Z\langle Y\rangle/L_d\hookrightarrow
O/K_d, \ \ f+L_d\mapsto f(x_1,\dotsc,x_n)+K_d,$$ Thus
$$\varprojlim Z\langle Y\rangle/L_d\hookrightarrow \varprojlim O/K_d.$$
Composing this embedding with the one given in
Lemma~\ref{lem:completion}, we obtain the embedding
$$\Phi\colon
Z\langle\langle Y\rangle\rangle\hookrightarrow \varprojlim Z\langle
Y\rangle / L_d\hookrightarrow \varprojlim O/K_d, \ \
\sum\limits_{\omega\in Y^*}z_\omega\omega\mapsto \sum_{\omega\in
Y^*}z_\omega \omega(x_1,\dotsc,x_n).$$

By condition (4), $1+x_i,\, (1+x_i)^{-1}\in O$ for $i=1,\dotsc,n$.
The natural ring homomorphisms $\pi_d\colon O\rightarrow O/K_d$ for
all $d\geq 1$, induce the ring homomorphism $\pi\colon O\rightarrow
\varprojlim O/K_d$. Note that $\pi(1+x_i)=\Phi(1+X_i)$, and
$\pi((1+x_i)^{-1})=\Phi((1+X_i)^{-1})$. By
Theorem~\ref{theo:MagnusFox}, the $Z$-subalgebra of $\varprojlim
O/K_d$ generated by $\{\pi(1+x_i),\pi(1+x_i)^{-1}\}_{i=1}^n$ is the
free group $Z$-algebra on $\{\pi(1+x_i)\}_{i=1}^n$. The facts that $\pi$ is
a $Z$-algebra homomorphism and $O$ is a subalgebra of $R$ imply the
result.
\end{proof}

Let $K$ be a field, and suppose that we want to find free group $K$-algebras on a
$K$-algebra $R$. Suppose $R$ is endowed with a map $\upsilon \colon R\rightarrow G_\infty$
satisfying (PF.1), (PF.2), (PF.3) and the restriction to $\upsilon$ to $F$ is a valuation.
Let $Z=\{a\in Z\colon \upsilon(a)\geq 0\}$. We can try to
apply Theorem~\ref{theo:freegroupalgebrafiltrationreals}, 
Theorem~\ref{theo:freegroupalgebrafiltrationorderedgroup} or 
Theorem~\ref{coro:divisionrings}
to the $Z$-algebra $R$. Then the
$F$-subalgebra of $R$ generated by $F$ and any free group $Z$-algebra is a free group $K$-algebra
because $F$ is the field of fractions of $Z$.

\medskip


\section{When the value group is a general ordered group}\label{sec:valuegeneralordered}

\subsection{Background on ordered groups}\label{sec:preliminariesordered}

We gather together some definitions and known results on
ordered groups. Everything can be found in
\cite[Sections~IV.1, IV.3]{Fuchs}, for example.

Given two ordered groups $(G,<)$ and $(H,\prec)$, a
\emph{homomorphism of ordered groups} is a homomorphism of groups
$\pi\colon G\rightarrow H$ such that $\pi(g)\preceq\pi(g')$ for all
$g,g'\in G$ such that $g\leq g'$.

\medskip

Let $(G,<)$ be an ordered group. The \emph{absolute} $|g|$ of an
element $g\in G$ is defined as $|g|=\max\{g,\, -g\}$. Let $g,h\in
G$. The element $g$ is said to be \emph{infinitely small relative
to} $h$ if $$n|g|< |h| \textrm{ for all positive integers } n;$$ we
denote this situation by $g\ll h$. On the other hand,  $g$ and $h$
are called \emph{Archimedean equivalent}, denoted by $a\sim b$, if
there exist positive integers $m$ and $n$ such that
$$|g|<m|h| \textrm{ and } |h|<n|g|.$$
It follows readily that
\begin{enumerate}[(i)]
\item for each pair $g,h\in G$ one, and only one, of the following
relations holds: $$g\ll h,\qquad g\sim h,\qquad h\ll g;$$
\item $g\ll h$ and $g\sim p$ imply $p\ll h$;
\item $g\ll h$ and $h\sim q$ imply $g\ll q$;
\item $g\ll h$ and $h\ll p$ imply $g\ll p$;
\item $g\sim h$ and $h\sim p$ imply $g\sim p$.
\end{enumerate}

\vspace{0.3cm}

Let $(G,<)$ be an ordered group. A subgroup $C$ of $G$ is
\emph{convex} if, for all $g,h\in C$ and $d\in G$, the inequality
$g\leq d\leq h$ implies that $d\in C$.

Let $\Sigma$ be the set of convex subgroups of $(G,<)$. It satisfies
the following properties:
\begin{enumerate}[(1)]
\item $\Sigma$ is a chain of subgroups of $G$ containing the trivial subgroups $\{0\}$
and $G$.

\item If  $\{C_\lambda\}_{\lambda\in\Lambda}\subseteq
\Sigma$, then their intersection $\bigcap_{\lambda\in\Lambda}
C_\lambda$ and union $\bigcup_{\lambda\in\Lambda} C_\lambda$ also
belong to $\Sigma.$

\item By a \emph{convex jump} $(N,C)$ in $\Sigma$ we mean a pair of subgroups
$N,C\in\Sigma$ such that $C$ properly contains $N$  and $\Sigma$
contains no subgroup between $C$ and $N$. In this event, $N$ is
normal in $C$. Moreover, the quotient group $C/N$ can be ordered in
the natural way:  for all $c_1,c_2\in C$ with $c_1+N\neq c_2+N$,
then $c_1+N<c_2+N$ if, and only if, $c_1<c_2$. With this ordering,
the group $C/N$ is order isomorphic to a subgroup of the real
numbers. Note that the natural map $C\rightarrow C/N$ is a
homomorphism of ordered groups, i.e. if $c_1\leq c_2$ then
$c_1+N\leq c_2+N$.

\item Every element $g\neq 0$ in $G$ defines a convex jump $(N_g,C_g)$ in
$\Sigma$ where $C_g$ is the intersection of all the elements of
$\Sigma$ that contain $g$ and $N_g$ is the union of all the elements
of $\Sigma$ which do not contain $g$. Note that $C_g=\{h\in G\colon
-n|g|\leq h\leq n|g|\textrm{ for some } n\geq 1\}$.
\end{enumerate}

\subsection{Main result}

The following lemma is similar to some results in
\cite[Chapter~1]{Schilling} which are obtained when $\upsilon$ is a valuation and $N=0$.

\begin{lem}\label{lem:valuationtothereals}
Let $(G,<)$ be an ordered group
and $N\subsetneq C$ be convex subgroups of $G$. Let $Z$ be a commutative ring and
$R$ be a $Z$-algebra. Let $\upsilon\colon R\rightarrow G_\infty$ be
a positive filtration.  The
following hold true.

\begin{enumerate}[\rm (1)]
\item The set $O=\{f\in R\colon \upsilon(f)\geq n \textrm{ for some } n\in
N\}$ is a subring of $R$.
\item The set $T=\{f\in O\colon \upsilon(f)>c \textrm{
for all }c\in C\}$ is an ideal of $O$.
\item The ring $O/T$ is a $Z/Z\cap T$-algebra and the map
$$\vartheta\colon O/T\rightarrow C_\infty, \quad \vartheta(f+T)=
\left\{\begin{array}{ll} \upsilon(f), & \textrm{if } f\in O\setminus
T
\\ \infty, & \textrm{if } f\in T\end{array}\right.$$
is a positive filtration (over $Z/Z\cap T$) with values in $C$.
\end{enumerate}
\end{lem}

\begin{proof}
(1) It follows easily from the definition of a positive filtration and
the fact that $N$ is a convex subgroup of $G$.

(2) We claim that if $g\in G$ is such that $g>c$ for all $c\in C$,
and $h\in G$ such that $h\geq n$ for some $n\in N$, then $g+h>c$ and
$h+g>c$ for all $c\in C$. The proof of the claim is divided into two cases whether 
$h\in N$ or $h\notin N$.
Indeed, if $h\in N$, then $g>c-h$ and
$g>-h+c$ for all $c\in C$ and therefore $g+h>c$ and $h+g>c$ for all
$c\in C$. Since $N$ is convex, if $h\notin N$, then $h>0$ and
$g+h>g>c$ and $h+g>g>c$ for all $c\in C$.

Let now $f_1,f_2\in T$ and $f\in O$. Using the foregoing claim,
$$\upsilon(f_1f)\geq \upsilon(f_1)+\upsilon(f)> c,\quad
\upsilon(ff_1)\geq \upsilon(f)+\upsilon(f_1)> c,$$ for all $c\in C$.
On the other hand, $\upsilon(f_1+f_2)\geq
\min\{\upsilon(f_1), \upsilon(f_2)\}>c$ for all $c\in C$. This shows
that $T$ is an ideal of $O$.

(3) Let $f_1,f_2\in O$. We have to prove that $\vartheta$ is well
defined and it satisfies the conditions in the definition of
positive filtration.

First we prove that $\vartheta$ is well defined. Suppose that
$f_1-f_2\in T$. Clearly, if $f_1\in T$, then $f_2\in T$,
$f_1+T=f_2+T=0+T$, and $\vartheta(f_1+T)=\vartheta(f_2+T)=\infty$.
Suppose now that $f_1\in O\setminus T$. Then $f_2=f_1+t$ with $t\in
T$ and $f_2\in O\setminus T$ too. Hence
$\upsilon(f_1),\upsilon(f_2)\in C$ and $\upsilon(f_1)<\upsilon(t)$.
Then
$\upsilon(f_2)=\min\{\upsilon(f_1),\upsilon(t)\}=\upsilon(f_1)$.

Clearly, $\vartheta(f+T)=\infty$ if, and only if, $f+T=0+T$.

Let $f_1,f_2\in O$. If $f_1f_2\in T$, then clearly,
$$\infty=\vartheta(f_1f_2+T)\geq\vartheta(f_1+T)+\vartheta(f_2+T).$$
Hence suppose that $f_1f_2\in O\setminus T$. It implies that
both $f_1,f_2$ do not belong to $T$. Thus
$$\vartheta(f_1f_2+T)=\upsilon(f_1f_2)\geq\upsilon(f_1)+\upsilon(f_2)=\vartheta(f_1+T)+\vartheta(f_2+T),$$
because $\upsilon$ is a positive filtration.

If $f_1+f_2\in T$, then $\vartheta
((f_1+f_2)+T)=\infty\geq\min\{\vartheta(f_1+T),\vartheta(f_2+T)\}$.
If $f_1+f_2\in O\setminus T$, then $f_1$ or $f_2$ belongs to
$O\setminus T$. Thus $\vartheta((f_1+f_2)+T)=\upsilon(f_1+f_2)\geq
\min\{\upsilon(f_1),\upsilon(f_2)\}=\min\{\vartheta(f_1+T),\vartheta(f_2+T)\}$.

Clearly $Z\subseteq O$ and $O/T$ is a $Z/Z\cap T$-algebra and
$\vartheta(z+T)\geq 0$ for all $z\in Z$. 


\end{proof}

It is not difficult to prove the following result whose version for valuations 
appears in \cite[Section~3]{FerreiraFornaroliSanchez}

\begin{lem}\label{lem:valuationhomomorphismordered}
Let $(G,<)$ and $(H,\prec)$ be ordered groups. Let $Z$ be a
commutative ring and $R$ be a $Z$-algebra with a positive filtration
$\upsilon\colon R\rightarrow G_{\infty}$. If $\pi\colon G\rightarrow
H$ is a homomorphism of ordered groups, then the map
$$\chi\colon R\rightarrow H_\infty, \quad \chi(f)=\left\{
\begin{array}{ll} \pi(\upsilon(f)), &  \textrm{if } f\in R\setminus\{0\} \\
\infty, & \textrm{if } f=0 \end{array}\right.$$ is a
positive filtration with values in $H$. \qed
\end{lem}

\medskip

\begin{theo}\label{theo:freegroupalgebrafiltrationorderedgroup}
Let $Z$ be a commutative ring and $R$ be a $Z$-algebra. Let $(G,<)$
be an ordered group and $\upsilon\colon R\rightarrow G_\infty$ be a
 positive filtration. Let $X$ be a subset of elements of $R$ satisfying the
following conditions
\begin{enumerate}[\rm (1)]
\item The map $X\rightarrow \gr_\upsilon(R)$, $x\mapsto x+R_{>\upsilon(x)}$, is injective.
\item For each $x\in X$, $\upsilon(x)>0$.
\item For each $x,y\in X$, $\upsilon(x)\sim \upsilon(y)$.
\item There does not exist $z\in Z$ such that $\upsilon(z)\gg
\upsilon(x)$ for some (and hence all) $x\in X$.
\item The subring of $\gr_\upsilon(R)$ generated by $\gr_\upsilon(Z)$ and the set
$\{x+R_{>\upsilon(x)}\}_{x\in X}$ is the free $\gr_\upsilon(Z)$-ring
on the set $\{x+R_{>\upsilon(x)}\}_{x\in X}$.
\item For each $x\in X$, $1+x$ is invertible in $R$ and $\upsilon((1+x)^{-1})\ll \upsilon(x)$.
\end{enumerate}
Then the $Z$-subalgebra of $R$ generated by $\{1+x,\, (1+x)^{-1}
\}_{x\in X}$ is the free group $Z$-algebra on the set
$\{1+x\}_{x\in X}$.
\end{theo}

\begin{proof}
Consider the chain $\Sigma$ of convex subgroups of $G$ and let
$(N,C)$ be the convex jump in $\Sigma$ determined by some $\upsilon(x_0)$, $x_0\in X$.
Notice that it is the same as the one determined by any
$\upsilon(x)$, $x\in X$, by condition (2). Thus $x\in C\setminus
N$ for all $x\in X$.

Observe that it is enough to show the result for $X$  a finite set.
Thus suppose that $X=\{x_1,\dotsc,x_n\}$.
Set $$O=\{f\in R\colon \upsilon(f)\geq n \textrm{ for some }n\in
N\},\ \ T=\{f\in O\colon \upsilon(f)> c \textrm{ for all } c\in
C\}.$$ Clearly $Z\subseteq O$.  By Lemma~\ref{lem:valuationtothereals}, $O$ is a subring of
$R$ and $T$ is an ideal of $O$. Set $S=O/T$. By (3),
$Z\cap T=0$. Thus $S$ is a $Z$-algebra.

The natural projection $\pi\colon C\rightarrow C/N$ is a
homomorphism of ordered groups where we identify $C/N$ with  an
additive subgroup of the real numbers. By
Lemmas~\ref{lem:valuationtothereals} and
\ref{lem:valuationhomomorphismordered}, the map $\chi\colon
S\rightarrow \mathbb{R}_\infty$, defined by
$$\chi(f+T)=\left\{\begin{array}{ll} \pi(\upsilon(f)), & \textrm{
if } f\in O\setminus T \\
\infty, & \textrm{ if } f\in T
\end{array}\right.$$ is a positive filtration (over $Z$).

We write $\widetilde{x}_i=x_i+R_{>\upsilon(x_i)}\in\gr_\upsilon(R)$ and
$\overline{x}_i=x_i+T\in S$ for each $i\in \{1,\dotsc,n\}$. Since there exists a  homomorphism of
$Z$-algebras $O\rightarrow S$, it is enough to prove that the
$Z$-subalgebra of $S$ generated by $\{1+\overline{x}_i,
(1+\overline{x}_i)^{-1}\}_{i=1}^n$ is the free group $Z$-algebra on
$\{1+\overline{x}_i\}_{i=1}^n$. For that, it is enough to verify
that the subset $\overline{X}=\{\overline{x}_i\}_{i=1}^n$ of $S$
satisfies conditions (1),(2),(3),(4) of
Theorem~\ref{theo:freegroupalgebrafiltrationreals}.

We claim that the map $X\rightarrow \gr_\chi(S)$, $x_i\mapsto \bar{x_i}+S_{>\chi(\bar{x_i})}$,
is injective. Notice that the map $X\rightarrow \gr_\upsilon(R)$, $x_i\mapsto x_i+R_{>\upsilon(x_i)}$,
is injective if, and only if, $\upsilon(x_i-x_j)=\min\{\upsilon(x_i),\upsilon(x_j)\}$
for $i\neq j$. Hence 
$\chi(\bar{x}_i-\bar{x}_j)=\pi\upsilon(x_i-x_j)=\pi(\min\{\upsilon(x_i),\upsilon(x_j)\})=
\min\{\pi\upsilon(x_i),\pi\upsilon(x_j)\}=\min\{\chi(\bar{x}_i),\chi(\bar{x}_j)\}$
which is equivalent to our claim.

Now note that $\chi(x_i)=\pi(\upsilon(x_i))\geq 0$. Moreover,
since $\upsilon(x_i)\notin N$, $\chi(x_i)>0$ for all $i=1,\dotsc,n.$ Also
observe that condition (5) implies that $(1+x_i)^{-1}\in O$ and
$\upsilon((1+x_i)^{-1})\in N$. Thus
$\chi((1+\overline{x}_i)^{-1})=0$ for each $i$.

Let $Y=\{X_1,\dotsc,X_n\}$ be a set of $n$ noncommuting variables.
Let $r$ be a real number. Let $\omega_1,\dotsc,\omega_l$ be distinct
words in $Y^*$ and $z_1,\dotsc,z_l\in Z$ be such that
\begin{equation}\label{eq:weighteddegree}
\chi(z_j\omega_j(\overline{x}_1,\dotsc,\overline{x}_n))=r \textrm{
for } j=1,\dotsc,l.
\end{equation}
We claim that
\begin{equation}\label{eq:claim} \chi\left(\sum_{j=1}^l
z_j\omega_j(\overline{x}_1,\dotsc,\overline{x}_n)\right)=r.
\end{equation}
Reordering the summands if necessary, we may suppose that the
summands in \eqref{eq:weighteddegree} satisfy
$$\upsilon(z_1\omega_1(x_1,\dotsc,x_n))\leq
\upsilon(z_2\omega_2(x_1,\dotsc,x_n))\leq
\dotsb\leq\upsilon(z_l\omega_l(x_1,\dotsc,x_n)).$$ Suppose that
$1\leq t\leq l$ is the maximum positive integer such that
$\upsilon(z_1\omega_1(x_1,\dotsc,x_n))=
\upsilon(z_2\omega_2(x_1,\dotsc,x_n))=
\dotsb=\upsilon(z_t\omega_t(x_1,\dotsc,x_n))$. By hypothesis (4),
$$0\neq \sum_{j=1}^t \widetilde{z}_j\omega_j(\widetilde{x}_1,\dotsc,\widetilde{x}_n)\in
R_{\upsilon(z_1\omega_1(x_1,\dotsc,x_n))}.$$ By
Lemma~\ref{lem:valuationsum}(3),
\begin{equation*}
\upsilon\Big( \sum_{j=1}^t z_j\omega_j(x_1,\dotsc,x_n)\Big)=\upsilon
(z_1\omega_1(x_1,\dotsc,x_n)).
\end{equation*}
Therefore
\begin{eqnarray*}
\upsilon\Big(\sum_{j=1}^l z_j\omega_j(x_1,\dotsc,x_n)\Big) & = &
\upsilon\Big(\sum_{j=1}^t z_j\omega_j(x_1,\dotsc,x_n)+
\sum_{j=t+1}^l z_j\omega_j(x_1,\dotsc,x_n)\Big) \\ & =&
\upsilon(z_1\omega_1(x_1,\dotsc,x_n)).
\end{eqnarray*}
Now, by definition, $$\chi\left(\sum_{j=1}^l
\widetilde{z}_jw_j(\widetilde{x}_1,\dotsc,\widetilde{x}_n)\right)=
\pi\Big(\upsilon\Big(\sum_{j=1}^l
z_jw_j(x_1,\dotsc,x_n)\Big)\Big)=r,$$ and the claim is proved.

From the last claim and using Lemma~\ref{lem:valuationsum}, it is not
difficult to show that the $\gr_\chi(Z)$-algebra  of $\gr_\chi(S)$
generated by $\{\overline{x}_i+S_{>\chi(\overline{x}_i)}\}_{i=1}^n$
is the free $\gr_\chi(Z)$-algebra on
$\{\overline{x}_i+S_{>\chi(\overline{x}_i)}\}_{i=1}^n$.
\end{proof}

\begin{rem} Theorem~\ref{theo:freegroupalgebrafiltrationorderedgroup} is
a generalization of
Theorem~\ref{theo:freegroupalgebrafiltrationreals}. Indeed, when $G$
is an additive subgroup of $\mathbb{R}$ with the induced natural ordering, conditions (3) and (4) in
Theorem~\ref{theo:freegroupalgebrafiltrationorderedgroup} can be
removed because they are trivially satisfied. Moreover, condition
$\upsilon((1+x^{-1})^{-1})\ll\upsilon(x)$ means that
$\upsilon((1+x)^{-1})=0$ for each $x\in X$. Thus we obtain
Theorem~\ref{theo:freegroupalgebrafiltrationreals}.
\end{rem}

\section{Division rings with valuations}\label{sec:divisionrings}

Our best results are obtained when $D$ is a division ring and
$\upsilon$ a valuation.

\begin{theo}\label{coro:divisionrings}
Let $D$ be a division ring with prime subring $Z$. Let $(G,<)$ be an ordered group and $\upsilon\colon
D\rightarrow G_\infty$ be a nontrivial valuation. Let
$X$ be a subset of $D$ satisfying the following
four conditions.
\begin{enumerate}[\rm (1)]
\item The map $X\rightarrow\gr_\upsilon(D)$, $x\mapsto x+D_{>\upsilon(x)}$, is injective.
\item For each $x\in X$, $\upsilon(x)>0$.
\item For all $x,y\in X$, $\upsilon(x)\sim\upsilon(y)$.
\item The $Z_0$-subalgebra of $\gr_\upsilon(D)$ generated by the set
$\{x+D_{>\upsilon(x)}\}_{x\in X}$ is the free $Z_0$-algebra
on the set $\{x+D_{>\upsilon(x)}\}_{x\in X}$, where 
$Z_0\coloneqq Z_{\geq 0}/Z_{>0}\subseteq D_0$.
\end{enumerate}
Then, for any central subfield $k$, the following hold true.
\begin{enumerate}[\rm (a)]
\item If there does not exists $z\in Z$ such that
$\upsilon(z)\gg\upsilon(x)$ for all $x\in X$, then the
$k$-subalgebra of $D$ generated by $\{1+x,\, (1+x)^{-1}\}_{x\in
X}$ is the free group $k$-algebra on $\{1+x\}_{x\in X}$.
\item If there exists $z\in Z$ such that $\upsilon(z)\gg
\upsilon(x)$ for all $x\in X$, then the $k$-subalgebra of $D$
generated by $\{1+zx,\, (1+zx)^{-1}\}_{x\in X}$ is the free
group $k$-algebra on the set $\{1+zx\}_{x\in X}$.
\end{enumerate}
\end{theo}
\begin{proof}
By hypothesis, the set $X$ satisfies conditions (1), (2), (3)
of
Theorem~\ref{theo:freegroupalgebrafiltrationorderedgroup}. 
Since
$\upsilon(x)>0$  for all $x\in X$, then $x\neq -1$ and
$\upsilon(1+x)=0$. Thus $1+x$ is invertible and
$0=\upsilon((1+x)^{-1})\ll \upsilon(x)$ for all $x\in X$. Hence condition
(6) is also satisfied. 

We claim that (4) is equivalent to Theorem~\ref{theo:freegroupalgebrafiltrationorderedgroup}(5).
Since $\upsilon$ is a valuation, $\gr_\upsilon(R)$ is a domain. 
 Let $S=\gr_\upsilon(Z)\setminus \{0\}$, and
localize $\gr_\upsilon(D)$ at $S$. Consider the natural embedding  
$\gr_\upsilon(D)\hookrightarrow S^{-1}\gr_\upsilon(D)$. Let $Q_0$
be the subfield of $C=S^{-1}\gr_\upsilon(Z)$ generated by $Z_0$. By Lemma~\ref{lem:MalcolmsonMakarLimanov}, the
$C$-subalgebra of $S^{-1}\gr_v(D)$ generated by a subset $Y$ is the free $C$-algebra on $Y$ if, and only if,
the $Q_0$-subalgebra of $S^{-1}\gr_v(D)$ generated by  $Y$ is the free $Q_0$-algebra on $Y$. Now observe that these 
subalgebras are free over $C$ (respectively $Q_0$) if, and only if, they are free over $\gr_\upsilon(Z)$ (respectively $Z_0$). 
Thus the claim is proved.

(a) Our additional hypothesis is Theorem~\ref{theo:freegroupalgebrafiltrationorderedgroup}(4).
Hence the   $Z$-subalgebra of $D$ generated by
$\{1+x,\, (1+x)^{-1}\}_{x\in X}$ is the free group $Z$-algebra
on $\{1+x\}_{x\in X}$. If the characteristic of $D$ is a prime,
then $Z$ is already the prime subfield. If the characteristic of $D$
is zero, then $Z$ can be identified with the integers and the prime
subfield of $D$ with the field of rational numbers. Moreover, the
$Z$-linear independence of the elements of the free group imply the
$\mathbb{Q}$-linear independence of these elements over the prime
subfield. Thus the $\mathbb{Q}$-subalgebra of $D$ generated by
$\{1+x,\, (1+x)^{-1}\}_{x\in X}$ is the free group
$\mathbb{Q}$-algebra on $\{1+x\}_{x\in X}$. By
Lemma~\ref{lem:MalcolmsonMakarLimanov}, we obtain that the $k$-subalgebra of
$D$ generated by $\{1+x,\, (1+x)^{-1}\}_{x\in X}$ is the free
group $k$-algebra on $\{1+x\}_{x\in X}$.

(b) First notice
that (b) can only happen when
the characteristic of $D$ is zero. 
By Lemma~\ref{lem:valuationintegers}, $\upsilon(z)>0$
and there does not exist $z'$ with $\upsilon(z')\gg \upsilon(z)$. By the results of Section~\ref{sec:preliminariesordered}
$\upsilon(zx)\sim\upsilon(zx')$ for all $x,x'\in X$. Now it can be shown
that   the elements of the set $\{zx\}_{x\in X}$ satisfy
conditions (1), (2), (3), (4) and (a) of the statement.
\end{proof}

Note that when $(G,<)$ is a subgroup of the real numbers, condition (3) in 
Theorem~\ref{coro:divisionrings}
is trivially satisfied. Moreover, 
condition (3) in Theorem~\ref{coro:divisionrings} is used to make
explicit the free generators of the free group algebra. In most
cases, there exists a free group algebra even if it is not
satisfied, as the next remark shows.

\begin{rem}
Let $D$ be a division ring with prime subring $Z$. Let $(G,<)$ be an
ordered group and $\upsilon\colon D\rightarrow G_\infty$ be a
nontrivial valuation. Let $X$ be a subset of $D$
satisfying the following four conditions.
\begin{enumerate}[\rm (1)]
\item The map $X\rightarrow \gr_\upsilon(D)$, $x\mapsto x+D_{>\upsilon(x)}$, is injective.
\item There exists $x_0\in X$, $\upsilon(x_0)>0$.
\item There does not exists $x\in X$ such that $\upsilon(x_0)\ll
\upsilon(x)$.
\item The $Z_0$-subalgebra of $\gr_\upsilon(D)$ generated by the set
$\{x+D_{>\upsilon(x)}\}_{x\in X}$ is the free $Z_0$-algebra
on the set $\{x+D_{>\upsilon(x)}\}_{x\in X}$
\end{enumerate}
Then $D$ contains a free group $k$-algebra on a set $Y$ with
$|Y|=|X|$ for any central subfield $k$ of $D$
\end{rem}

\begin{proof}
By (2) and (3), there exist nonnegative integers $n_x$ such that
$\upsilon(xx_{0}^{n_x})>0$ and $\upsilon(xx_{0}^{n_x})\sim
\upsilon(x_{0})$ for each $x\in X$. Then (4) implies that the
$Z_0$-subalgebra of $\gr_\upsilon(D)$ generated by the set
$\{xx_{0}^{n_x}+D_{>\upsilon(xx_{0}^{n_x})}\}_{x\in X}$ is
the free $Z_0$-algebra on the set
$\{xx_{0}^{n_x}+D_{>\upsilon(xx_{0}^{n_x})}\}_{x\in X}$. Now
Theorem~\ref{coro:divisionrings} implies the existence of free
group $k$-algebras on a set $Y=\{xx_0^{n_x}\}_{x\in X}$ with $|Y|=|X|$.
\end{proof}

\begin{coro}\label{coro:freemonoid}
Let $D$ be a division ring. Let $(G,<)$ be
an ordered group and $\upsilon\colon D\rightarrow G_\infty$ be a
surjective valuation. Let $g,h\in G$ 
such that the submonoid of $G$ generated by $\{g,h\}$ is the free monoid on
$\{g,h\}$.
Then $D$ contains a free group $k$-algebra for any central subfield
$k$ of $D$.
\end{coro}

\begin{proof}
In order to ease the notation, the group $G$ will be multiplicative during this proof.

{\underline{Step~1:}} There exist elements $a,b\in G$ such that
$1<a,b$, $a\sim b$ and they generate a free monoid of rank two.

Consider $|g|$, $|h|$. Suppose that $|g|<|h|$. We suppose that
$1<h$, otherwise the following argument works with $g^{-1}$,
$h^{-1}$ because the free monoid on $\{g^{-1},h^{-1}\}$ is also free
on $\{g^{-1},h^{-1}\}$.

Either $g<1$ or $1<g$. Suppose that $1<g$. Then $g\sim h$ or $g\ll
h$. If $g\sim h$ then $1<a=g,b=h$ and $\{a,b\}$ freely generates a free
monoid of rank two. If $g\ll h$, then $1<gh=a\sim h=b$ and $\{a,b\}$ freely
generates a free monoid of rank two. Suppose now that $g<1$. Note that both
$gh$ and $h$ are $>1$. By the foregoing, either $\{a=gh,b=h\}$ generate a free
monoid of rank two, or $\{a=gh^2,b=h\}$ generate a free monoid of
rank two.

{\underline{Step~2:}} Let $Z$ be the prime subring of $D$. Let
$x,y\in D$ with $\upsilon(x)=a$, $\upsilon(y)=b$ where $a,b\in G$ are as in Step~1. Then the
$Z_1$-subalgebra of $\gr_\upsilon(D)$ generated by $\{x+D_{>a},\,
y+D_{>b}\}$ is the free $Z_1$-algebra on $\{x+D_{>a},\,
y+D_{>b}\}$.

Let $w_1=w_1(X,Y),\, w_2=w_2(X,Y)$ be two different words on the
noncommutative variables $X,Y$. Then $\upsilon(w_1(x,y))=w_1(a,b)$
and $\upsilon(w_2(x,y))=w_2(a,b)$. Since $\{a,b\}$ generates a free
monoid of rank two, $w_1(a,b)\neq w_2(a,b)$. Moreover, for any
$z_1,z_2\in Z$ with $\upsilon(z_1)=\upsilon(z_2)=1$, $$\upsilon(z_1w_1(x,y))=\upsilon(z_1)w_1(a,b)\neq
\upsilon(z_2)w_2(a,b)=\upsilon(z_2w_2(x,y)).$$  Hence, if we denote, $\widetilde{d}=d+D_{>\upsilon(d)}$ for
any $d\in D$, we get that, for any $z_1,\dotsc,z_s\in
Z\setminus\{0\}$ with $\upsilon(z_i)=1$ and different words $w_1(X,Y),\dotsc,w_s(X,Y)$,
$$\widetilde{z}_iw_i(\widetilde{x},\widetilde{y})=z_iw_i(x,y)+D_{>\upsilon(z_iw_i(x,y))},\quad i=1,\dotsc,s,$$
and therefore $$0\neq \sum_{i=1}^s
\widetilde{z}_iw_i(\widetilde{x},\widetilde{y})\in\gr_\upsilon(D),$$ as desired.

{\underline{Step~3:}} Proof of the result.

Thus the subset $\{x,y\}\subseteq D$ in Step~2 satisfies conditions
(1),(2),(3),(4) in Theorem~\ref{coro:divisionrings}. Let $k$ be a
central subfield of $D$. Applying
Theorem~\ref{coro:divisionrings}, we obtain that either the
$k$-algebra generated by $\{1+x,(1+x)^{-1},1+y,(1+y)^{-1}\}$ is the
free group $k$-algebra on the set $\{1+x,1+y\}$, or there exists
$z\in Z$ such that the $k$-subalgebra of $D$ generated by
$\{1+zx,(1+zx)^{-1},1+zy,(1+zy)^{-1}\}$ is the free $k$-algebra on
the set $\{1+zx,1+zy\}$.
\end{proof}

We would like to note that the existence of free monoids in ordered groups
often happens. Indeed, by \cite[Lemma~8]{LongobardiMajRhemtullanofree}, an
ordered group $(G,<)$ that contains a convex subgroup which is not normal
contains a free monoid on two generators. This implies that  if the orderable group $G$
has no free submonoid on two generators, then all the convex subgroups are normal in
$G$ under any order on $G$ (\cite[p.1420]{LongobardiMajRhemtullanofree}).

Also, if all the convex subgroups of $(G,<)$ are normal, but 
there exists a convex jump $(N,C)$ which is not central, that is, $[C,G]\nsubseteq N$,
then \cite[Section~3.2]{SanchezfreegroupalgebraMNseries} implies the existence of a free monoid in $G$.

Corollary~\ref{coro:freemonoid} is a generalization of \cite[Lemma~2.3]{SanchezfreegroupalgebraMNseries}. 
In that result it is considered the Malcev-Neumann series $K((G;<))$ ring obtained from an ordered group $(G,<)$
and a division ring $K$ endowed with its natural valuation $\upsilon\colon K((G;<))\rightarrow G_\infty$.

\medskip

An instance where our results can be applied is to what is known as microlocalization. 
Let $R$ be a ring with a valuation $\upsilon\colon R\rightarrow \mathbb{Z}_\infty$.
Consider the graded ring $\gr_\upsilon(R)$ and the subset $S$ of $\gr_\upsilon(R)$
consisting of the nonzero homogeneous elements. If $S$ is an Ore subset of $\gr_\upsilon(R)$, then $R$ can 
be embedded in a division ring $D$ such that $\upsilon$ extends to a valuation
$\upsilon\colon D\rightarrow\mathbb{Z}_\infty$ with $\gr_\upsilon(D)=S^{-1}\gr_\upsilon(R)$
and such that $D$ is complete with respect to the topology induced by $\upsilon$, see \cite{VandenEssen}, 
\cite{AsensioVandenBerghVanOystaeyen}, \cite{Lichtmanvaluationmethods} (note
that the filtration in those references is ascending). 
Note that the filtration in those papers is an ascending one.
In some important cases where this process of microlocalization is applied,
the graded ring is commutative,  which is not good for our purpouse of finding  free group algebras because
the graded ring certainly will not contain free algebras.

\medskip

Let $D$ be a division ring and $\upsilon\colon D\rightarrow \mathbb{Z}_\infty$
be a surjective valuation. Note that $D_{\geq 0}$ is a local ring with maximal
ideal $D_{>0}$, thus $D_0=D_{\geq 0}/D_{>0}$ is a division ring. 
Then $\gr_\upsilon(D)$ is isomorphic to a skew Laurent polynomial ring
$D_0[t,t^{-1};\alpha]$ where $t$ is any fixed nonzero element of $D_1=D_{\geq 1}/D_{>1}$.
To obtain a free group algebra in $D$, it is enough to find different elements
$at^m$, $bt^n$ with $m,n\geq 1$ that generate a free algebra over a subfield of $D_0$
fixed by $\alpha$. There are some references that deal with the problem of finding such
free algebras, for example \cite[Section~9C]{Lam2}, \cite[Proposition~2]{LichtmanmatrixringsII},
\cite{PaulSmith1}, \cite{PaulSmith2}.


\appendix
\section{A counterexample}\label{sec:counterexample}

In this section we present a counterexample to
\cite[Proposition~4]{LichtmanmatrixringsII}.

\medskip

Let $G$ be a group and $A,B$ be  subgroups of $G$. By $[A,B]$ we
denote the subgroup of $G$ generated by $\{a^{-1}b^{-1}ab\colon a\in
A,\, b\in B\}$. We define $\gamma_1(G)=G$, and for $i\geq 1$,
$\gamma_{i+1}=[G,\gamma_i(G)]$. Also $\sqrt{A}=\{g\in G\colon g^n\in
A \textrm{ for some } n\geq 1\}$.

Let $F$ be the free group on two letters $x,y$. Let
$F''=[\gamma_1(F),\gamma_1(F)]$. Let $M=F/F''$ be the free
metabelian group on two letters $u=xF''$ and $v=yF''$. Consider the
group ring $\mathbb{Q}[M]$, where $\mathbb{Q}$ denotes the field of
rational numbers.

In \cite{Moufangmetabelian}, it was proved that the submonoid of $M$
generated by $u$ and $v$ is the free monoid on the set $\{u,v\}$.
Thus
\begin{enumerate}[$\qquad (\spadesuit)$]
\item\label{counterexample1} the $\mathbb{Q}$-subalgebra of $\mathbb{Q}[M]$ generated by
$\{u-1,\, v-1\}$ is the free {$\mathbb{Q}$-algebra} on the set
$\{u-1,\, v-1\}$.
\end{enumerate}

Consider the augmentation map
$\varepsilon\colon\mathbb{Q}[M]\rightarrow \mathbb{Q}$ and let
$\omega(\mathbb{Q}[M])=\ker \varepsilon$ be the \emph{augmentation
ideal}. For each $i\geq 1$, define the \emph{dimension subgroups} of
$M$ as
$$D_i(\mathbb{Q}[M])=\{x\in M\colon x-1\in\omega(\mathbb{Q}[M])^i\}.$$
Since $\mathbb{Q}$ is a field of characteristic zero,
\cite[Chapter~11, Theorem~1.10]{Passman2} implies that
$D_i(\mathbb{Q}[M])=\sqrt{\gamma_i(M)}$ for each $i\geq 1$. By
\cite[Theorem~D2]{HartleyResidualnilpotence}, $M$ is a residually
torsion-free nilpotent group. Thus $\bigcap\limits_{i\geq
1}\sqrt{\gamma_i(M)}=\bigcap\limits_{i\geq
1}D_i(\mathbb{Q}[M])=\{1\}$. This implies that
$\bigcap\limits_{i\geq 1}\omega(\mathbb{Q}[M])^i=0$ by
\cite[Chapter~11, Theorem~1.21]{Passman2}.

The powers of the ideal $\omega(\mathbb{Q}[M])$ define a filtration
of $\mathbb{Q}[M]$ $$\mathbb{Q}[M]\supseteq
\omega(\mathbb{Q}[M])\supseteq\omega(\mathbb{Q}[M])^2\supseteq
\dotsb\supseteq \omega(\mathbb{Q}[M])^i\supseteq
\omega(\mathbb{Q}[M])^{i+1}\supseteq \dotsb$$ Consider the graded
ring associated to this filtration
$$\gr(\mathbb{Q}[M])=\bigoplus_{i\geq 0}\frac{\omega(\mathbb{Q}[M])^i}{\omega(\mathbb{Q}[M])^{i+1}},$$
where we assume $\omega(\mathbb{Q}[M])^0=\mathbb{Q}[M]$.

Consider now, for each $i\geq 1$, the torsion-free abelian groups
$D_i(\mathbb{Q}[M])/D_{i+1}(\mathbb{Q}[M])$. The direct sum
$$L(M)=\bigoplus_{i\geq 1}\frac{D_i(\mathbb{Q}[M])}{D_{i+1}(\mathbb{Q}[M])}$$
can be endowed with a Lie $\mathbb{Z}$-algebra structure, see for
example \cite[VIII.2]{Passibookaugmentation}. Consider the Lie
$\mathbb{Q}$-algebra $Q\otimes_\mathbb{Z}L(M)$ and its universal
enveloping algebra $U(Q\otimes_\mathbb{Z}L(M))$. By
\cite{Quillengraded} or \cite[VIII,
Theorem~5.2]{Passibookaugmentation}, there exists an isomorphism of
$\mathbb{Q}$-algebras $\Theta\colon
U(Q\otimes_\mathbb{Z}L(M))\rightarrow \gr(\mathbb{Q}[M])$ such that
if $x+D_{i+1}(M)\in D_i(\mathbb{Q}[M])/D_{i+1}(\mathbb{Q}[M])$, then
$$\Theta(x+D_{i+1}(M))=(x-1)+\omega(\mathbb{Q}[M])^{i+1}\in \frac{\omega(\mathbb{Q}[M])^i}{\omega(\mathbb{Q}[M])^{i+1}}.$$
Since $U(Q\otimes_\mathbb{Z}L(M))$ is a domain, $\gr(\mathbb{Q}[M])$
is a domain. Now \cite[Proposition~2.6.1]{Cohnskew}, implies that
the map $\chi\colon\mathbb{Q}[M]\rightarrow \mathbb{Z}_\infty$
defined by $\chi(f)=\sup\{i\colon f\in\omega(\mathbb{Q}[M])^i\}$ is
a valuation. Observe that, since
$u'=u-1,v'=v-1\in\omega(\mathbb{Q}[M])$,
\begin{enumerate}[$\qquad (\spadesuit\spadesuit)$]
\item\label{counterexample2} $\chi(u')\geq 1$ and $\chi(v')\geq 1$.
\end{enumerate}

Since $M$ is an extension of torsion-free abelian groups,
$\mathbb{Q}[M]$ is an Ore domain. Let $D$ be the Ore division ring
of fractions of $\mathbb{Q}[M]$. Then the valuation $\chi$ can be
extended to a valuation on $D$, see for example
\cite[Proposition~9.1.1]{Cohnskew}. Therefore we have a valuation
$\chi\colon D\rightarrow \mathbb{Z}_\infty$ such that $u',v'\in D$,
the $\mathbb{Q}$-subalgebra of $D$ generated by $\{u',v'\}$ is the
free $\mathbb{Q}$-algebra on $\{u',v'\}$ by $(\spadesuit)$. By
$(\spadesuit\spadesuit)$, $\chi(u')\geq 1$, $\chi(v')\geq 1$. But
$1+u'=u$ and $1+v'=v$ do not generate a free group
$\mathbb{Q}$-algebra on $\{1+u',\, 1+v'\}$ because $u$ and $v$ do
not even generate a free group.


\section{}\label{sec:corrigenda}

In the following results, the existence of free algebras in certain division rings was proved. Then 
\cite[Proposition~4]{LichtmanmatrixringsII} was used to show the existence of a free group algebra in certain division rings.
They may well be true, but the proof provided depends on \cite[Proposition~4]{LichtmanmatrixringsII}, and our 
Theorem~\ref{coro:divisionrings} cannot be applied.

\begin{enumerate}[(1)]
\item \cite[Corollary~1 of Proposition~4]{LichtmanmatrixringsII}
\item \cite[Corollaries~2.1,2.2,Theorems~B,C,D]{FigueiredoGoncalvesShirvani}
\item \cite[Corollary B, Theorems C,D,E]{ShirvaniGoncalvesdifferential}
\item \cite[Theorem~1]{ShirvaniGoncalvesLarge}
\item \cite[Theorem~5]{Lichtmanfreeuniversalenveloping}
\item \cite[Theorem~4]{GoncalvesTengan}
\item \cite{GoncalvesShirvaniSurvey} is a survey on the existence of free objects in division rings where some of 
the foregoing  results are stated.
\item \cite[Theorem~1.1,Proposition~3.3,Theorem~3.5]{FerreiraGoncalvesSanchez1}. In these results the
existence of a free algebra generated by symmetric elements is proved.
\item \cite[Corollary~4.10]{Fehlberg}
\end{enumerate}

\medskip

In the following results, the existence of free group algebras
in division rings was proved using \cite[Proposition~4]{LichtmanmatrixringsII}, but with the help of other results and methods
they can be reproved, at least partially.

\begin{enumerate}[(1)]
\item \cite[Corollary~2 of Proposition~4]{LichtmanmatrixringsII} was proved and generalized in \cite{SanchezObtaininggraded}.
Hence, as noted in \cite[Comments~3.1(2)]{SanchezObtaininggraded}, 
\cite[Lemma~7.5]{HerberaSanchez}, 
\cite[Corollary~5.4]{Herberasanchezinfiniteinversion} and \cite[Proposition~3.1]{FerreiraGoncalvesSanchez1} 
are correct because they relied on [15, Corollary 2 of Proposition~4]. 
Furthermore, the proof of \cite[Proposition 3.1]{FerreiraGoncalvesSanchez1} can be simplified. The result
follows by applying \cite[Corollary~3.4]{SanchezObtaininggraded} to \cite[Eq. (4), p. 78]{FerreiraGoncalvesSanchez1}.
 
\item The proof of \cite[Proposition~3.4,Proof of Theorem~3.1 for ordered groups of type~1]{SanchezfreegroupalgebraMNseries}
used \cite[Corollary~2.1]{FigueiredoGoncalvesShirvani}.
This result was needed to obtain a free group algebra in the Ore field of fractions of the group algebra $k[G]$
of a nonabelian torsion-free nilpotent group $G$ over a field $k$. 
This free group algebra can be obtained by \cite[Th\'eor\`eme]{Cauchoncorps}
except for the characteristic two case.
\end{enumerate}

\bibliographystyle{amsplain}
\bibliography{grupitosbuenos}

\end{document}